\documentclass[11pt]{article}
\usepackage{geometry}
\usepackage{graphicx}	
\usepackage[cp1251]{inputenc}
\usepackage[english]{babel}
\usepackage{mathtools}
\usepackage{amsfonts,amssymb,mathrsfs,amscd,amsmath,amsthm}
\usepackage{verbatim}

\def\thtext#1{
  \catcode`@=11
  \gdef\@thmcountersep{. #1}
  \catcode`@=12
}

\def\threst{
  \catcode`@=11
  \gdef\@thmcountersep{.}
  \catcode`@=12
}

\theoremstyle{plain}
\newtheorem{thm}{Theorem}
\newtheorem{prop}{Proposition}[section]
\newtheorem{cor}[prop]{Corollary}
\newtheorem{ass}[prop]{Assertion}
\newtheorem{lem}[prop]{Lemma}

\theoremstyle{definition}

\newtheorem{rk}[prop]{Remark}

\newtheorem{agree}[prop]{Agreement}

 \catcode`@=11
 \def\.{.\spacefactor\@m}
 \catcode`@=12

\renewcommand{\a}{\alpha}

\newcommand{\D}{\Delta}
\newcommand{\e}{\varepsilon}

\newcommand{\g}{\gamma}
\newcommand{\G}{\Gamma}

\newcommand{\om}{\omega}
\newcommand{\Om}{\Omega}
\renewcommand{\r}{\rho}
\newcommand{\s}{\sigma}

\renewcommand{\t}{\tau}
\renewcommand{\v}{\varphi}
\newcommand{\Y}{\Upsilon}

\newcommand{\cA}{\mathcal{A}}
\newcommand{\cB}{\mathcal{B}}
\newcommand{\cD}{\mathcal{D}}
\newcommand{\cE}{\mathcal{E}}
\newcommand{\cG}{\mathcal{G}}
\newcommand{\cL}{\mathcal{L}}
\newcommand{\cM}{\mathcal{M}}
\newcommand{\cN}{\mathcal{N}}
\newcommand{\cO}{\mathcal{O}}
\newcommand{\cP}{\mathcal{P}}
\newcommand{\cQ}{\mathcal{Q}}

\newcommand{\cBT}{\mathcal{B\!T}}

\newcommand{\N}{\mathbb{N}}

\newcommand{\R}{\mathbb{R}}
\newcommand{\Z}{\mathbb{Z}}

\newcommand{\rom}[1]{{\em #1}}

\renewcommand{\)}{\rom)}

\renewcommand{\:}{\colon}
\newcommand{\0}{\emptyset}

\renewcommand{\>}{\rangle}
\renewcommand{\c}{\circ}
\renewcommand{\d}{\partial}

\newcommand{\oPi}{\stackrel{\raise-2pt\hbox{$\c$}}\Pi}
\newcommand{\oW}{\stackrel{\raise-2pt\hbox{$\c$}}W}
\newcommand{\sm}{\setminus}
\renewcommand{\sp}{\supset}
\renewcommand{\ss}{\subset}

\newcommand{\im}{{\operatorname{im}\,}}
\newcommand{\Len}{{\operatorname{Len}}}
\newcommand{\mf}{{\operatorname{mf}}}
\newcommand{\mpf}{{\operatorname{mpf}}}

\newcommand{\smt}{{\operatorname{smt}}}
\newcommand{\SMT}{\operatorname{SMT}}
\newcommand{\mpn}{{\operatorname{mpn}}}
\newcommand{\MPN}{\operatorname{MPN}}

\def\ig#1#2#3#4{\begin{figure}[!ht]\begin{center}%
\includegraphics[height=#2\textheight]{#1.eps}\caption{#4}\label{#3}%
\end{center}\end{figure}}

\title{Analytic Deformations of Minimal Networks}
\author{A.~Ivanov \and A.~Tuzhilin}
\date{}							

\begin{document}
\maketitle

\begin{abstract}
A behavior of extreme networks under deformations of their boundary sets is investigated. It is shown that analyticity  of a deformation of boundary set guarantees preservation of the networks types for minimal spanning trees, minimal fillings and so-called stable shortest trees in the Euclidean space.
\end{abstract}

\section*{Introduction}
\markright{Introduction}
An interest to the behavior of shortest networks under deformations of their boundary sets goes back to the works devoted to Steiner ratio of Euclidean plane investigations. In 90th, in Australian school, some non-trivial estimates on the Steiner ratio for small boundary sets were obtained by means of so-called variational approach, i.e., using a control on behavior of the length functions of minimal spanning tree and shortest tree under small deformations of a boundary set (which are sometimes referred as variations), see~\cite{RubThom}, \cite{Thom}, and~\cite{Weng}.

Investigation of the length functions of extreme networks were continued by the authors. In paper~\cite{ITCalc}, for shortest trees and for minimal spanning trees in the Euclidean space it is proved an existence of the derivative of their length functions with respect to one-parametrical deformations of the boundary sets, and formulas for those derivatives were derived. Namely, it turns out that the derivative is equal to the derivative of some minimal parametric tree (a tree with a fixed combinatorial structure) which is a shortest one (a minimal spanning one, respectively) for the initial boundary.

But it is not true that for a sufficiently small deformation of the boundary a structure of shortest tree or of minimal spanning tree can be chosen changeless. A corresponding example can be easily constructed even for a three-point boundary set in the Euclidean plane, see Figure~\ref{fig:examp}. One fixed vertex $O$ of the triangle is located in the origin, the other fixed vertex $A$ belongs to the ray forming the angle of $2\pi/3$ with the abscissa axis, and the third vertex $B$ is located at the abscissa axis in the initial time and moves along an oscillating curve which makes infinitely many oscillations around the abscissa axis in any neighborhood of the initial point. When the point $B$ is in the upper half-plane, then the shortest network for the set $AOB$ has an additional vertex, where three its edges meet each other, and when $B$ is in the lower half-plane, then the shortest network for $AOB$ consists of two segments, $[A,O]$ and $[O,B]$.

\ig{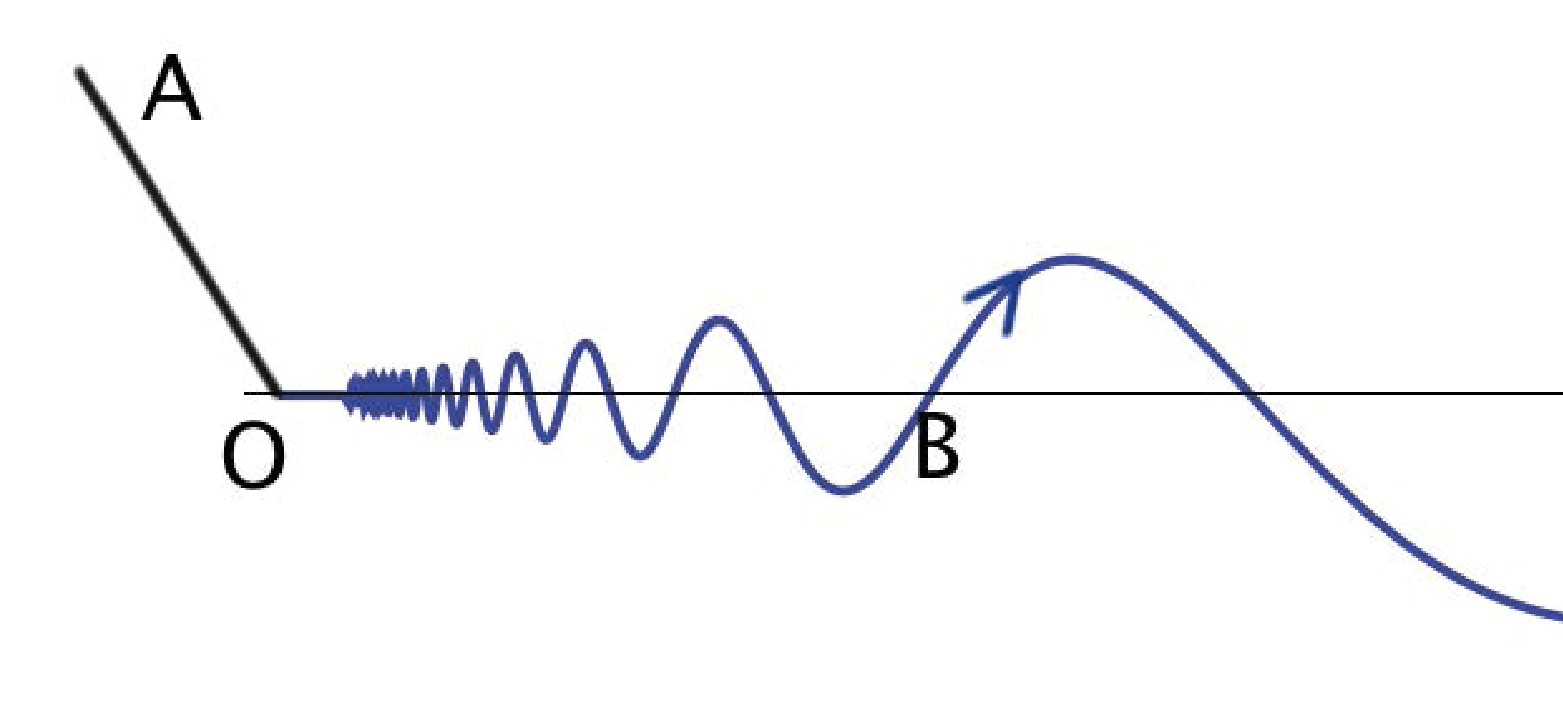}{0.25}{fig:examp}{A smooth deformation generating infinitely many changes of the shortest tree structure.}

This example can be modified easily to generate infinitely many changes of non-degenerate shortest tree structures. Consider a convex quadrangle in the plane, and let $O$ be the intersection point of its diagonals. Assume that the vertices of this quadrangle can be connected by two locally minimal binary trees (here binary tree is a tree, whose boundary vertices have degree $1$, and all the other ones have degree $3$). For each such tree consider the pair of vertical angles having vertex at $O$ and subtending the sides connecting adjacent boundary vertices. Then, in accordance with~\cite{Pollak}, the tree corresponding to the least angle is shorter, and hence, is a shortest one, i.e., it is a Steiner minimal tree (and if the angles are equal o each other, then the both trees are shortest).

Take a square as the quadrangle. Then its vertices are connected by two shortest binary trees, see Figure~\ref{fig:examp2}. Let us move one of the vertices of the square similarly to the previous example, namely, its trajectory oscillates around  the straight line passing through the correspondent diagonal of the square. Then the shortest tree connecting the four point set consisting of the three fixed vertices of the square and the moving one changes its type each time as the moving vertex passes across the straight line (the combinatorial structure of the binary tree remains the same, but the way the tree is attached to the boundary set changes spasmodically). 

Notice that both our examples are based on the same effect, namely, on an infinite number of oscillations in a finite interval of the parameter changing. The deformation constructed are smooth but not analytic.

\ig{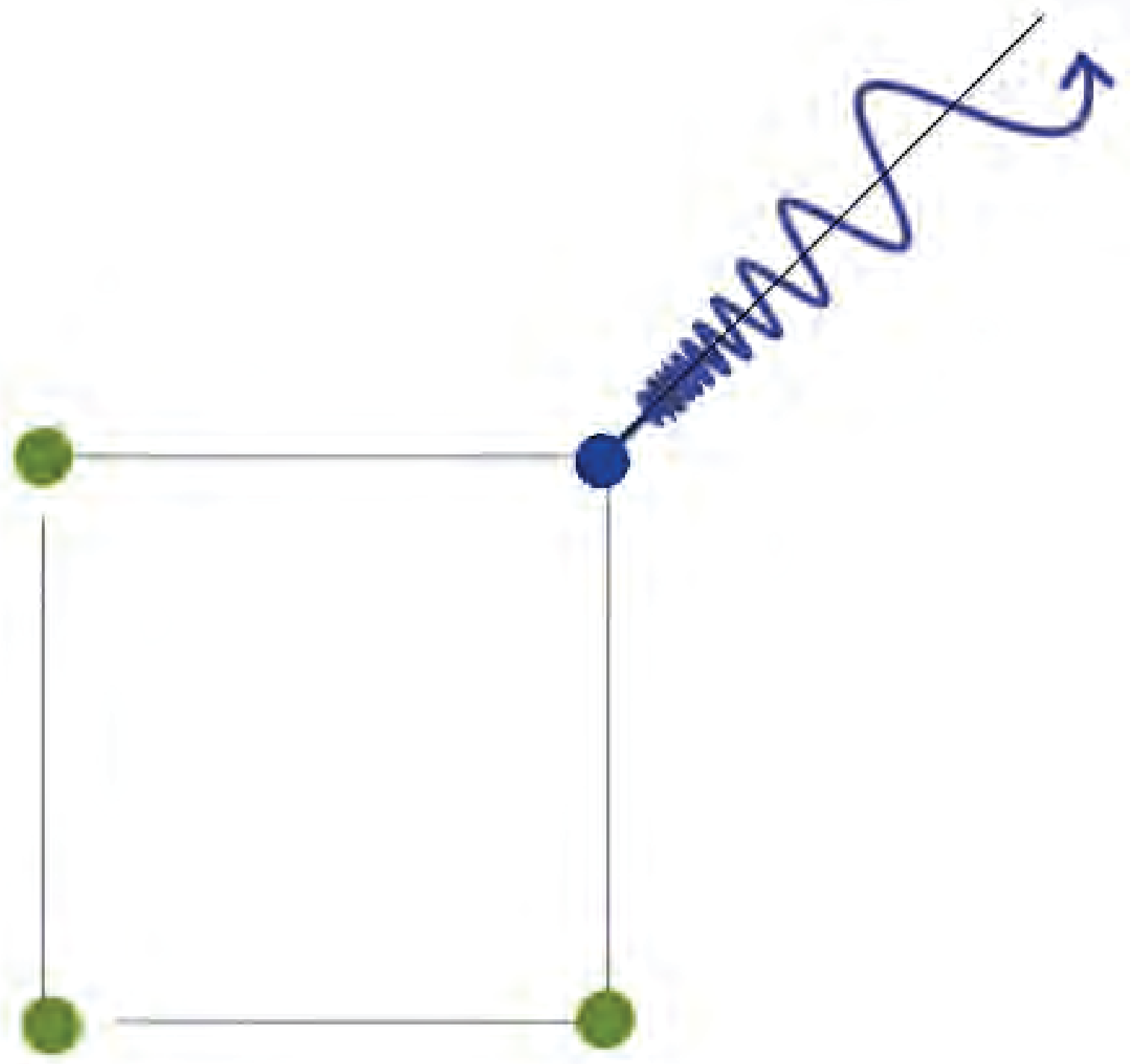}{0.25}{fig:examp2}{A smooth deformation of the boundary generating infinitely many changes of non-degenerate shortest tree structures.}

The main result of the present paper describes a class of deformations which such effects are impossible under. Main Theorem (Theorem~\ref{thm:main}) gives an answer for a sufficiently general case (see below a general statement of the problem in terms of so-called metric functionals). As applications, the classical cases of shortest trees and minimal spanning trees in Euclidean space are considered. It is shown (see Corollary~\ref{cor:Euc_case_mst}) that  {\em if under a one-parametric  deformation $W_t$, $t\in[0,1]$, each point of an initial boundary set $W_0$ moves along an analytic curve, then for sufficiently small positive values of the parameter $t$ the families $\cN_t$ of the combinatorial types of the minimal spanning trees with the boundary $W_t$ are the same and contained in $\cN_0$}. To prove similar result for shortest trees we need an additional assumptions of so-called stability, see below, of all shortest trees with the initial boundary $W_0$, see Corollary~\ref{cor:smt_gen}. The case of minimal fillings of finite metric spaces is also investigated, see Corollaries~\ref{cor:mf_gen} and~\ref{cor:mf_eucl}. Besides, for convenience, we include well-known formulas of the first and the second variation of the length of a segment in the Euclidean space, see Assertions~\ref{ass:deriv1} and~\ref{ass:deriv2}.

\section{Basic Notations and General Construction}
\markright{\thesection.~Basic Notations and General Construction}
Let $V$ be some finite set, whose elements are referred as \emph{vertices}, and $X$ be an arbitrary set  which is called an \emph{ambient space}. By $\cM(V,X)$ we denote the set of all mappings from $V$ to $X$. Let a numeration $V=\{v_1,\ldots,v_n\}$ of the vertices be fixed. Then each mapping $f\in\cM(V,X)$ is uniquely defined by the ``vector'' $\bigl(f(v_1),\ldots,f(v_n)\bigr)$ of its values that defines natural isomorphism between the set $\cM(V,X)$ and the Cartesian power $X^n$.

Let $\cD(X)$ be the set of all semi-metrics on $X$, then for each $f\in\cM(V,X)$ and each $\r\in\cD(X)$ the semi-metric $f^*(\r)\in\cD(V)$ is defined as follows: $f^*(\r)(v_i,v_j)=\r\bigl(f(v_i),f(v_j)\bigr)$. Since each semi-metric from  $\cD(V)$ is uniquely defined by the set of its $m=n(n-1)/2$ values on all pairs of distinct vertices from $V$, then the space $\cD(V)$ can be identified with the subset $\cD^m$ of the Euclidean space $\R^m$ consisting of all the vectors with non-negative coordinates such that some of their triplets meet the triangle inequalities. In more detail,  $\cD^m\ss\R^m$ is the subset of points of the following form:
$$
\bigl(r_{12},\ldots,r_{1n},r_{23},\ldots,r_{2n},\ldots,r_{n-1\,n}\bigr),
$$
such that the inequalities $r_{ij}\ge0$ and $\bigl|r_{ij}-r_{jk}\bigr|\le r_{ik}\le r_{ij}+r_{jk}$ are valid for all $1\le i<j<k\le n$. Thus, the vector $f^*(\r)$ has the above form, where $r_{ij}=f^*(\r)(v_i,v_j)$.

Each function $L\:D\to\R$, where $D\ss\cD^m$, is referred as a  \emph{metric functional}. A functional $L$ is said to be \emph{continuous}, if  $L$ is a continuous function. A functional $L$ is said to be \emph{smooth\/  \(analytic\/\)}, if it is a restriction of a smooth function (an analytic function, respectively) defined in an open subset $\Om\subset\R^m$, $\Om\sp D$. Similarly, \emph{continuous}, \emph{smooth\/} and \emph{analytic curves in $D\ss\cD^m$} are defined.

By a \emph{curve in a set $W$ passing through a point $w\in W$ in instant $t_0\in[a,b]$} we call a mapping $\g\:[a,b]\to W$ such that $\g(t_0)=w$. If $t_0=a$, then one says that the curve $\g$ \emph{goes out $w$}, and if $t_0=b$, then $\g$ \emph{comes to $w$}. A \emph{deformation $f_t$} of a mapping $f\in\cM(V,X)$ is a curve in  $\cM(V,X)$ passing through $f$. For each metric $\r\in\cD(X)$ the deformation $f_t$ generates the \emph{deformation $f_t^*(\r)$ of the metric $f^*(\r)$} which is a curve in $\cD^m$ passing through the point $f^*(\r)\in\cD^m$.

Let $\cL=\{L^1,\ldots,L^p\}$ be a non-empty family of metric functionals, and $D\ss\cD^m$ be the intersection of their domains. Then 
$$
\cL_{\min}=\min_i L^i\qquad \text{and}\qquad \cL_{\max}=\max_i L^i
$$
are defined on $D$. For each $r\in D$ by  $I_{\min}(\cL,r)$ we denote the set of the indices $i$ such that $\cL_{\min}(r)=L^i(r)$. Similarly, by $I_{\max}(\cL,r)$ we denote the set of the indices $i$ such that $\cL_{\max}(r)=L^i(r)$.

\section{Analytic Functions and Main Theorem}
\markright{\thesection.~Analytic Functions and Main Theorem}
The following result from Analytic Functions Theory is key one in our considerations.

\begin{prop}\label{prop:two_analitic_functions}
Let $f(x)$ and $g(x)$ be real functions defined in an open domain in the real line $\R$ analytic at a point  $x_0$, and $f(x_0)=g(x_0)$. Then there exists a neighborhood $U$ of the point $x_0$ such that either $f$ and $g$ coincide with each other in it, or $f(x)\ne g(x)$ for all $x\in U\sm\{x_0\}$.
\end{prop}

The following general result follows directly from the above definitions and Proposition~\ref{prop:two_analitic_functions}.

\begin{thm}\label{thm:main}
Let $\cL=\{L^1,\ldots,L^p\}$ be a family of metric functionals having a nonempty intersection $D\ss\cD^m$ of their domains. Consider an arbitrary mapping $f\in\cM(V,X)$ and its deformation  $f_t$, $t\in[a,b]$, where  $f=f_{t_0}$, $t_0\in [a,b]$. Let = $\r\in\cD(X)$ be an arbitrary semi-metric. Consider the corresponding semi-metric $f^*(\r)\in\cD^m$ and the corresponding curve $\g(t)=f_t^*(\r)$ in $\cD^m$ generating a deformation of the latter semi-metric. Let the curve $\g$ lie in $D$, so the set $I_{\min}(t)=I_{\min}\bigl(\cL,\g(t)\bigr)$ is defined for all  $t\in[a,b]$. Besides, let $\g$ be analytic at $t_0$, and all the functionals $L^i$ be analytic at $\g(t_0)$. Then
\begin{enumerate}
\item There exists a neighborhood $U\ss[a,b]$ of the point $t_0$, such that the sets $I_{\min}(t)$ are the same for all $t\in U\cap\{t>t_0\}$ \(for all $t\in U\cap\{t<t_0\}$, respectively\/\) and they are contained in $I_{\min}(t_0)$\/\rom;
\item if also $I_{\min}(t)=I_{\min}(t_0)$ for some $t\in U\sm\{t_0\}$, then the sets $I_{\min}(t)$ are the same for all $t\in U$.
\end{enumerate}
\end{thm}

\begin{rk}
A similar result is valid for $I_{\max}(t)=I_{\max}\bigl(\cL,\g(t)\bigr)$ also.
\end{rk}

We also need the following analytic version of Implicit Function Theorem, see, for example~\cite{Horman}.

\begin{prop}\label{prop:explicit_analitic_fucntion}
Let $f_j(z,w)$, $j=1,\ldots,m$, be a set of real functions on real variables  $(z,w)=(z^1,\ldots,z^m,w^1,\ldots,w^n)$, which are analytic in a neighborhood of a point $(z_0,w_0)$, and let $f_j(z_0,w_0)=0$, $j=1,\ldots,m$, and
$$
\det\Bigl(\frac{\d f_j}{\d z^k}\Bigr)_{j,\,k=1}^m\ne0\ \text{at the point $(z_0,w_0)$} .
$$
Then in some neighborhood of the point $w_0$ the equation system  $f_j(z,w)=0$, $j=1,\ldots,m$, has uniquely defined analytic solutions $z^k(w)$, $k=1,\ldots,m$, such that $z(w_0)=z_0$.
\end{prop}

\section{Main Theorem in the Case of Networks}
\markright{\thesection.~Main Theorem in the Case of Networks}

For an arbitrary set $W$ by $\cP(W)$ we denote the set of all its subsets, and let  $W^{(k)}\ss\cP(W)$ stand for the set of all $k$-element subsets of $W$.

By an \emph{edge set $E$} on a set $V$ we call an arbitrary subset of $V^{(2)}$. Elements $e=\{v,w\}\in E$ are referred as \emph{edges\/} and are denoted by $vw$ or $wv$ for brevity. Each edge set $E$ generates the corresponding metric functional $L_E$ on $\cD^m$ as follows:  if $E=\0$, then put $L_E=0$, otherwise, 
$$
L_E(r_{12},\ldots,r_{n-1\,n})=\sum_{v_iv_j\in E}r_{ij},
$$
where $r_{12},\ldots,r_{n-1\,n}$ are Cartesian coordinates in $\R^m$. It is clear that the functional $L_E$ is analytic.

A pair $G=(V,E)$, where  $E\ss V^{(2)}$, is called by a (simple) \emph{graph}.  A \emph{generalized network of the type $G$ in $X$} is a mapping $\G\:V\to X$, and a deformation $\G_t$ of the mapping $\G$ is a \emph{deformation of the generalized network $\G$}. The number $L_E\bigl(\G^*(\r)\bigr)$ is called the \emph{length of the generalized network $\G$ with respect to the semi-metric $\r$} and is denoted by $\Len_\r(\G)$.

Let $\cE=\{E^1,\ldots,E^p\}$ be a non-empty family of edge sets on $V$, and $\cG=\bigl\{(V,E^1), \ldots, (V,E^p)\bigr\}$ be the corresponding family of the graphs $G^i=(V,E^i)$. Put $L^i=L_{E^i}$, and form the family $\cL=\{L^1,\ldots,L^p\}$ of functionals on $\cD^m$. Then $\cL_{\min}$ is called the \emph{length of minimal generalized network of the type $\cG$}.

Further, let $(X,\r)$ be a semi-metric space. Then for each $f\in\cM(V,X)$ the family of generalized networks $\cN=\cN(\cG)=\{\G^i\}$ is defined, where $\G^i$ stands for the network of the type $G^i$ that coincides with $f$ as a mapping from $V$ to $X$.  For each generalized network $\G^i$ from this family its length $\Len_\r(\G^i)$ is defined. Each generalized network  $\G^i$ such that $\Len_\r(\G^i)=\cL_{\min}\bigl(f^*(\r)\bigr)$ is called a  \emph{minimal generalized network in the family $\cN$}.

\begin{cor}\label{cor:general_nets_param}
Let $(X,\r)$ be an arbitrary semi-metric space, and $\cN(\cG)=\{\G^i\}$ be the family of generalized networks generated by a family of graphs $\cG$ and a mapping $f\in\cM(V,X)$. Consider a deformation $f_t$, $t\in[a,b]$, of these networks. Assume that the corresponding curve $\g(t)=f_t^*(\r)$ in $\cD^m$ is analytic at the point $t_0$. Then
\begin{enumerate}
\item There exists a neighborhood $U\ss[a,b]$ of the point $t_0\in[a,b]$ such that the types of minimal generalized networks in the family $\cN$ are the same for all  $t\in U\cap\{t>t_0\}$ \(for all  $t\in U\cap\{t<t_0\}$, respectively\/\) and they are contained among the types of minimal generalized networks in $\cN$ for $t=t_0$\/\rom;
\item If also the sets of types of minimal generalized networks in the family $\cN$ are the same for some $t\in U\sm\{t_0\}$ and $t_0$, then these sets are the same for all $t\in U$.
\end{enumerate}
\end{cor}

For example, if $\cG$ is the family of all spanning trees on the set $V$, then $\cL_{\min}$ is called the \emph{length of minimal spanning tree}, and each minimal generalized network in this family is referred as a  \emph{minimal spanning tree}.

\begin{cor}\label{cor:general_mst}
Let $(X,\r)$ be an arbitrary semi-metric space, and $f_t$, $t\in[a,b]$, be a deformation of some mapping $f\in\cM(V,X)$. Assume that the corresponding curve $\g(t)=f_t^*(\r)$ in $\cD^m$ is analytic at the point  $t_0$. Then
\begin{enumerate}
\item There exists a neighborhood $U\ss[a,b]$ of the point $t_0\in[a,b]$ such that the types of minimal spanning trees are the same for all $t\in U\cap\{t>t_0\}$ \(for all  $t\in U\cap\{t<t_0\}$, respectively\/\), and they are contained in the set of types of minimal spanning trees for $t=t_0$\/\rom;
\item If also the types of minimal spanning trees are the same for some $t\in U\sm\{t_0\}$ and $t_0$, then these types are the same for all $t\in U$.
\end{enumerate}
\end{cor}

\section{Networks in Euclidean Space}
\markright{\thesection.~Networks in Euclidean Space}
In this Section the ambient space is the space $\R^k$ endowed with the Euclidean distance $\rho_2$. Let  $V=\{v_1,\ldots,v_n\}$ be a finite set, and $f\:V\to\R^k$ be an embedding. Put $W=f(V)\subset\R^k$. If  $f(v_i)=w_i=(w_i^1,\ldots,w_i^k)$, $i=1,\ldots,n$, then
$$
f^*(\r_2)(v_i,v_j)=\r_2\bigl(f(v_i),f(v_j)\bigr)=\r_2(w_i,w_j)=\sqrt{\sum_\a(w_i^\a-w_j^\a)^2}.
$$
Each deformation $f_t$, $t\in[a,b]$, of the embedding $f$ generates the one-parametric family of sets $W_t=f_t(V)\subset\R^k$ which is defined by the set of curves $w_i(t)$, $i=1,\ldots,n$. In coordinates  the curve $\g(t)=f^*_t(\r_2)$ in $\cD(V)$ has the following form:
$$
r_{ij}(t)=f^*_t(\r_2)(v_i,v_j)=\r_2\bigl(w_i(t),w_j(t)\bigr)=\sqrt{\sum_\a\bigl(w_i^\a(t)-w_j^\a(t)\bigr)^2}.
$$

\begin{ass}\label{ass:Eucl_case}
Under the above notations, if all the curves $w_i(t)$ are analytic at $t_0$, and all the points $w_i(t_0)$, $i=1,\ldots,n$, are pairwise distinct, then the curve $\g(t)$ is analytic at the point $t_0\in[a,b]$.
\end{ass}

\subsection{Minimal Networks without Additional Vertices}
The next two results follows directly from Corollaries~\ref{cor:general_nets_param}, and~\ref{cor:Euc_case_mst}, and Assertion~\ref{ass:Eucl_case}.

\begin{cor}\label{cor:Euc_case_param_nets}
Let $\cN=\{\G^i\}_{i=1}^p$ be a family of generalized networks generated by a family $\bigl\{G^i=(V,E^i)\bigr\}$ of graphs and by a fixed embedding $f\:V\to\R^k$. Put $W=f(V)=\{w_1,\ldots,w_n\}$, and let $W_t$, $t\in[a,b]$, be a one-parametric deformation of the set $W$ given by a set of curves $w_i(t)$, $t\in[a,b]$. Assume that all the curves  $w_i(t)$ are analytic at the point $t_0\in[a,b]$ and that the points $w_i(t_0)$, $i=1,\ldots,n$, are pairwise distinct. Then
\begin{enumerate}
\item There exists a neighborhood $U\ss[a,b]$ of the point $t_0$ such that the types of minimal generalized networks in the family $\cN$ connecting the sets  $W_t$ are the same for all $t\in U\cap\{t>t_0\}$ \(for all  $t\in U\cap\{t<t_0\}$, respectively\/\) and they are contained among the types of the minimal generalized networks in the family $\cN$ connecting the set $W_{t_0}$\/\rom;
\item If also the sets of types of minimal generalized networks in the family $\cN$ connecting the sets $W_t$ and $W_{t_0}$ coincides withe each other for some $t\in U\sm\{t_0\}$, then these sets are the same for all  $t\in U$.
\end{enumerate}
\end{cor}

\begin{cor}\label{cor:Euc_case_mst}
Let $W_t$, $t\in[a,b]$, be a one-parametric deformation of a finite subset $W=\{w_1,\ldots,w_n\}\subset\R^k$, oven by a set of curves $w_i(t)$, $t\in[a,b]$. Assume that all the curves $w_i(t)$ are analytic at the point $t_0\in[a,b]$, and that the points $w_i(t_0)$, $i=1,\ldots,n$, are pairwise distinct. Then
\begin{enumerate}
\item There exists a neighborhood $U\ss[a,b]$ of the point $t_0$ such that the types of minimal spanning trees spanning the sets $W_t$ are the same for all $t\in U\cap\{t>t_0\}$ \(for all $t\in U\cap\{t<t_0\}$, respectively\/\) and they are contained among the types of minimal spanning trees connecting $W_{t_0}$\/\rom;
\item If also the types of minimal spanning trees spanning $W_t$ and $W_{t_0}$ are the same for some $t\in U\sm\{t_0\}$, then these sets are the same for all  $t\in U$.
\end{enumerate}
\end{cor}

\begin{rk}
The above construction, Assertion~\ref{ass:Eucl_case} and Corollaries~\ref{cor:Euc_case_param_nets} and~\ref{cor:Euc_case_mst} can be generalized word-byword to the case of the space $\R^k$ endowed with the metric  $\r_p$, $1<p<\infty$, where
$$
\r_p(w_i,w_j)=\sqrt[p]{\sum_\a(w_i^\a-w_j^\a)^p}, \qquad w_i=(w_i^1,\ldots,w_i^k)\in\R^k.
$$
More general, the same results are valid for a normed space with a norm that is analytic everywhere except zero.
\end{rk}

\subsection{Necessary Information from Graph and Network Theory}
As above, a pair $G=(V,E)$, where $E\ss V^{(2)}$, is referred as a (simple) \emph{graph}. Since we are interested in boundary problems, we always assume that for each graph $G=(V,E)$ under consideration a set of its vertices that are referred as \emph{boundary\/} is chosen; the set of boundary vertices is the \emph{boundary of the graph $G$} which is denoted by $\d G$. Due to specifics of the boundary problem under consideration, we always assume that $\d G$ contains all the vertices of degree $1$ and $2$ of the graph $G$. The remaining non-boundary vertices of the graph $G$ are referred as  \emph{interior\/} or \emph{movable}.

Let $G=(V,E)$ be tree with a boundary $\d G$, and $v\in \d G$ be a boundary vertex of degree $d\ge2$.  Represent the tree $G$ as a union of its subtrees in such a way that the degree of the boundary vertex $v$ in each these subtree is equal to $1$. To do that  consider all the edges $e_i=u_iv$, $i=1,\ldots,d$, of the graph $G$ that are incident to $v$ and through out of the graph $G$ all the edges $e_i$ except some single $e_j$. By $G_j$ we denote the unique connected component of the resulting graph that contains the edge $e_j$. Evidently, $G_j$ is a subtree in $G$, and the degree of the vertex $v$ in it i.e. equal to one. Put $\d G_j=V_j\cap\d G$, where $G_j=(V_j,E_j)$. It is clear that the tree $G$ is the union of its subtrees $G_j$, $j=1,\ldots,d$. We say that the subtrees $G_j$ are obtained from $G$ by \emph{cutting it by the boundary vertex $v$}. If one cuts a tree $G$ consecutively by all its boundary vertices of degree more than one, then the boundaries of the resulting subtrees consist exactly of all their vertices of degree  $1$. These subtrees are referred as  \emph{regular components\/} of the tree $G$.

We call a tree  (with a boundary)  \emph{binary}, if all its vertices have degrees either $1$ or $3$, and its boundary consists exactly of all its vertices of degree $1$. A pair of adjacent edges of a binary tree, each of which is  incident to a bounder vertex is called  \emph{moustaches}. Each binary tree with three or more boundary vertices has moustaches.

Let $B$ be an arbitrary finite set. By $\cBT(B)$ we denote the set of all binary trees with the boundary $B$ considered up to an isomorphism preserving $B$. 

Let $G=(V,E)$ be a tree with a boundary $\d G$, and $E_d$ be an arbitrary family of edges of the tree $G$. By  $G_1,\ldots,G_k$ we denote the components of the forest $(V,E_d)$, by  $V_i$ we denote the vertex set of the tree  $G_i$, and put $W=\{V_i\}$. Then  $W$ is a partition of the set $V$ that is said to be \emph{generated by the family $E_d$}. By $\pi\:V\to W$ we denote the canonical projection, i.e.,  $\pi(v)=V_i$, if and only if $v\in V_i$. For any edge $e\in E\sm E_d$ its vertices belong to distinct sets $V_i$, therefore $\pi(e)\in W^{(2)}$. Since distinct $V_i$ and $V_j$ are connected by at most one edge of the tree $G$, then the mapping $\pi\:E\sm E_d\to W^{(2)}$ is injective. Put $F=\pi(E\sm E_d)$. It is easy to see that $H=(W,F)$ is a tree. Since each $V_i$ that is connected with other $V_j$ by at most two edges contains a vertex from $\d G$, then all the vertices of degree $1$ and $2$ of the tree $H$ belong to $\pi(\d G)$. Thus, the set $\pi(\d G)$ can be chosen as a  \emph{boundary $\d H$ of the tree  $H$}. The resulting tree $H$ with the boundary  $\d H$ is called  the \emph{quotient\/} of the tree $G$ or the result of  \emph{factorization\/} of the tree  $G$ over the family $E_d$ of edges. The quotient is denoted by  $G/E_d$.

\begin{rk}
The factorization operation defined above differs from the standard factorization over a subset, under which the whole subset turns into a single element: the vertex sets of distinct connected components $G_i$ go to  \emph{distinct\/} vertices of the quotient.
\end{rk}

If $E_d$ consists of a single edge $e$, then we say that $H$ is obtained from $G$ by \emph{degenerating\/} or by \emph{contracting of the edge $e$}, and $G$ is obtained from $H$ by \emph{splitting of the vertex $w$}, where $w\ni\pi(e)$. A factorization of a tree $G$ can be represented as a consecutive contracting of its edges; conversely, each tree $G$ can be repaired from its quotient by a consecutive splitting of its vertices.

Generally speaking, splitting of a vertex is defined ambiguously. Namely, the edges incident to the initial vertex  can be distributed between the new vertices in different ways, and if the initial vertex is a boundary one, then the resulting vertices can be differently distributed between boundary and interior vertices. In the latter case we always assume that t least one of the new vertices is referred to the boundary, and that all the new vertices of degree $1$ or $2$ are also boundary ones.

Each tree is a quotient of some binary tree that is ambiguously defined in general.

Let $G=(V,E)$ be an arbitrary tree, $X$ be a set, and $f\:V\to X$ be an arbitrary mapping. Transform the graph  $\G_f$ of the mapping $f$ into a combinatorial graph as follows: chose $\G_f$ as a vertex set, and connect $\bigl(u,f(u)\bigr)$ and $\bigl(v,f(v)\bigr)$ by an edge, if and only if $uv\in E$. As a result, we obtain a tree which is isomorphic to $G$ with the corresponding boundary. This resulting tree is called a  \emph{network in $X$ of the type $G=(V,E)$}.

\begin{agree}
For convenience of networks operating we make the following agreements:
\begin{itemize}
\item Identify the mapping $f$ and its graph $\G_f$ and denote them by the same symbol, say by $\G$;
\item Denote a vertex $\bigl(v,\G(v)\bigr)$ of the network $\G$ by $x_v$ and identify it with $\G(v)$ (compare with the notation $x_n$ for elements of a sequence of reals, which can be defined as a mapping from $\N$ to $\R$), so, even if $\G(u)=\G(v)$, but $u\ne v$, then the corresponding vertices $x_u$ and $x_v$ of the network $\G$ are considered as different ones; 
\item On the other hand, to distinct the sets $\{x_v\}_{v\in V}$ of the vertices of a network $\G$ and the corresponding subset of $X$, we denote by $\im\G$ the latter one (considering $\G$ as a mapping), in particular, $\d\G=\{x_v\}_{v\in\d G}$ and $\im\d\G=\G(\d G)$;
\item Identify edges $e=uv$ of the tree $G$ with the corresponding edges $x_ux_v$  of the network $\G$.
\end{itemize}
In fact, a network $\G$ is obtained from a tree $G$ by ``assignment'' a ``location'' $\G(v)$ in $X$ to each vertex $v\in V$ .
\end{agree}

An edge  $x_ux_v$ of a network $\G$ of a type  $G$ is called  \emph{degenerate}, if $\G(u)=\G(v)$. The corresponding edge of the tree $G$ is called \emph{$\G$-degenerate}. A network without degenerate edges is called  \emph{non-degenerate}.

Let $S$ be a subfamily of the set of degenerate edges of a network $\G$, and $H=G/S=(W,F)$, $W=\{V_i\}$, be the corresponding quotient. Since the mapping $\G$ maps each $V_i$ to a single point, then the mapping $\D\:W\to X$ such that $\D(V_i)=\G(v)$, where $v\in V_i$, is well-defined. The network $\D$ of the type $H$ is called the \emph{quotient of the network $\G$ with respect to the edge set $S$} and is denoted by $\G/S$. If $S$ is chosen to be the set of all degenerate edges of $\G$, then the network $\D=\G/S$ is called the \emph{trace of the network $\G$} and is denoted by $\tau(\G)$. By Definition the trace of an arbitrary network dose not contain degenerate edges. If $\cA$ is a family of networks, then by $\tau(\cA)$ we denote the set of all the traces of the networks from  $\cA$.

Let $\D=\tau(\G)$ be the trace of some network $\G$, and $H$ be the type of the network $\D$. By $\cB(\D)$ we denote the set of all binary trees $T$ that can be obtained from $H$ by splitting of vertices. Notice that the tree $H$ can be obtained from each such binary tree $T$ by factorization over some appropriate set of edges $S_T$. For each  $T\in\cB(\D)$ the network $\G_T$ of the type $T$, such that $\G_T/S_T=\D$ is uniquely defined. The set of degenerate edges of the network $\G_T$ coincides with $S_T$, so $\D=\tau(\G_T)$. The set $\cB(\D)$ is called the  \emph{binary type of the trace $\D$}.

We say that a network $\G$ in $X$ \emph{connects\/} a finite subset $M$ of the set $X$, if $\im\d\G=M$. By  $\cN(X,M)$ we denote the set of all the networks connecting $M$.

Each mapping $\v\:\d G\to X$ is referred as a \emph{boundary\/} one. Let some boundary mapping $\v$ be fixed. We say that a  \emph{network $\G$ connects a set $M\subset X$ by the mapping $\v$}, if $\G(v)=\v(v)$ for all $v\in\d G$, and $\im\d\G=M$. By $[G,\v]$ we denote the set of networks parameterized by a tree $G$ and connecting some set $M$ by a given mapping $\v$.

Now let $(X,\r)$ be a metric space, and $\G$ be some network. As in the case of a generalized network, by the \emph{length of an edge $x_ux_v$ of the network $\G$} we call the number $\r(x_u,x_v)$, i.e., the length between the vertices $x_u$ and $x_v$. The sum of lengths of all the edges of the network $\G$ is called the  \emph{length\/} of this network and is denoted by $\Len_\r(\G)$. Notice that if $\D$ is a quotient of the network  $\G$, then $\Len_\r(\G)=\Len_\r(\D)$.

Let  $M$ be a finite subset of $X$. Put 
$$
\smt(M)=\inf\bigl\{\Len_\r(\G)\mid \G\in\cN(X,M)\bigr\}.
$$
The number $\smt(M)$ is called the \emph{minimum length of networks on $M$}. Each network $\G\in\cN(X,M)$ such that $\Len_\r(\G)=\smt(M)$ is called a  \emph{shortest network}.

A shortest network can have degenerate edges. The trace of a shortest network has no degenerate edges, i.e., is a non-degenerate tree and called a \emph{Steiner minimal tree on $M$}. By $\SMT(M)$ we denote the set of all Steiner minimal trees on $M$.

Let $G$ be a fixed tree, and $\v\:\d G\to X$ be some boundary mapping. Put
$$
\mpn_G(\v)=\inf\bigl\{\Len_\r(\G)\mid\G\in[G,\v]\bigr\}.
$$
The number $\mpn_G(\v)$ is called the  \emph{minimum length of networks of the type $G$ with the boundary $\v$}. Each network  $\G\in[G,\v]$ such that $\Len_\r(\G)=\mpn_G(\v)$ is called a  \emph{minimal parametric network of the type $G$ with the boundary $\v$}. By $\MPN(G,\v)$ we denote the set of minimal parametric networks of a type $G$ with a boundary  $\v$.

The set $\SMT(M)$ of shortest networks for a given boundary $M$ can be empty. The set of all parametric networks of a given type with a fixed boundary an be also empty.

In what follows we need the following Assertions, see, for example, \cite{ITUMN} and~\cite{ITAMS93}, or book~\cite{ITCRC}.

\begin{ass}\label{ass:gen_bin_tree_reduction}
Let $M\ss X$ and $B$ be arbitrary sets consisting of $n$ elements, and $\v\:B\to M$ be some bijection. Then 
\begin{gather*}
\smt(M)=\min\bigl\{\mpn_G(\v)\mid G\in\cBT(B)\bigr\},\\
\SMT(M)=\bigcup_{\{G\in\cBT(B)\mid\mpn_G(\v)=\smt(M)\}}\t\bigl(\MPN(G,\v)\bigr).
\end{gather*}
\end{ass}

\begin{ass}\label{ass:Eucl_exist}
Let $M$ be a finite subset of the space $\R^k$ endowed with Euclidean metric, $G$ be a tree with a boundary $B$, and $\v\:B\to M$ be an arbitrary bijection. Then $\MPN(G,\v)$ is non-empty. If at the same time $G$ is a binary tree and the network $\G\in\MPN(G,\v)$ is non-degenerate, then $\G$ is unique minimal parametric network of the type  $G$ with the boundary $\v$, and the segments corresponding to adjacent edges of the network meet at a common vertex by the angle of $2\pi/3$. 
\end{ass}

\begin{ass}\label{ass:loc_structure}
Let $M$ be a finite subset of the space $\R^k$ endowed with Euclidean metric, and $\G$ be the trace of a shortest tree connecting $M$. Then the segments corresponding to adjacent edges of the network $\G$ meet at a common vertex by angle which is greater than or equal to $2\pi/3$. In particular, the degrees of vertices of the network  $\G$ does not exceed three, and at each vertex of degree three the angles between adjacent edges are equal to  $2\pi/3$.
\end{ass}

Let $\G$ be the trace of a minimal parametric network of of a type $G$, where $G$ is a binary tree from Assertion~\ref{ass:Eucl_exist}. The regular components $\{\G_1,\ldots,\G_m\}$ of the tree $\G$ are binary trees, distinct $\G_i$ and $\G_j$  intersect each other by at most one common boundary vertex. Each network $\G_i$ is a non-degenerate minimal parametric binary tree connecting $M_i$.  If two regular components meet at a boundary vertex $x_v$, then the angle between the segments corresponding to the edges incident with $x_v$ is greater than or equal to $2\pi/3$, and if there are three such components, then the corresponding angles are equal to $2\pi/3$. If the initial tree $\G$ is a shortest, then each its regular component $\G_i$ is a shortest tree connecting the corresponding $M_i$.

\subsection{Non-Degenerate Minimal Parametric Binary Trees and Non-Degenerate Shortest Trees}
Again consider the space $\R^k$ endowed with Euclidean distance $\rho_2$ as an ambient space. Let $B=\{v_1,\ldots,v_n\}$ be a finite set, and $\v\:B\to\R^k$ be an embedding. Put $M=\v(B)\subset\R^k$. Fix some binary tree $G\in\cBT(B)$. Put $G=(V,E)$ and $I=V\setminus B$. As it is known, see Assertion~\ref{ass:Eucl_exist}, there exists a minimal parametric tree $\G\in\MPN(G,\v)$ of the type $G$ with the boundary $\v$. Moreover, if all the edges of the tree $\G$ are non-degenerate, then it is unique, in other words in this case the location of non-boundary vertices of the tree $\G$ is uniquely defined. Therefore, the mapping $\v$ can be uniquely extended to the mapping $\G$ defined at the set $V$, and hence, a mapping $\Y$ arises that maps a set of boundary vertices  $M=(m_1,\ldots,m_n)\in\R^{nk}$ to the corresponding set $Z=\G(I)=(z_1,\ldots,z_{n-2})\in\R^{(n-2)k}$ of movable vertices of the minimal parametric network $\G$ (the mapping $\Y$ depends evidently on the enumeration of the vertices of the tree $G$). 

\begin{ass}\label{ass:int_vertices}
Assume that all edges of a minimal parametric binary tree $\G_0$ with a boundary $M_0$ are non-degenerate. Then the mapping $Z=\Y(M)$ is defined in a neighborhood of the point $M_0$ and is analytic at the point $M_0$.
\end{ass}

\begin{proof}
Since each network $\G\in[G,\v]$ is uniquely defined by the images $\G(s)$ of its movable vertices $s\in I$, i.e\. by the vector $Z=\G(I)=(z_1,\ldots,z_{n-2})\in\R^{(n-2)k}$, then the function $\ell_G(Z,M)$ of the length of the network  $\G$ with the set of interior vertices $Z$ and the set of boundary vertices $M$ is defined. By the assumptions, all the edges of the minimal tree $\G_0$ with the boundary $M_0$ are non-degenerate, therefore the function $\ell_G(Z,M)$ that is equal to the sum of the lengths of the corresponding straight segments is analytic at the point $(Z_0,W_0)$, where $Z_0$ corresponds to the interior vertices of the network $\G_0$ with the boundary $M_0$.  The minimal parametric network $\G_0$ with the boundary $\v$ is the unique extremum of the function $h(Z)=\ell_G(Z,M_0)$. Therefore, the locations of interior vertices $Z$ are uniquely defined by the condition of the partial derivatives $\d\ell_G/\d z_i^j$, $i=1,\ldots,n-2$, $j=1,\ldots,k$, vanishing, where $z_i=(z_i^1,\ldots,z_i^k)\in\R^k$ are the coordinates of the movable vertices.

\begin{lem}\label{lem:second_deriv}
The matrix 
$$
\Bigl(\frac{\d^2\ell_G}{\d z_i^j\d z_p^q}\Bigr)
$$
of the second derivatives is non-degenerate at the point $(Z_0,M_0)$.
\end{lem}

\begin{proof}
The function $h(Z)=\ell_G(Z,M_0)$ is the sum of the lengths of non-degenerate straight segments which are the edges of the parametric tree of the type $G$ with the boundary $M_0$ and the interior vertices $Z$. The point $Z_0$ corresponding to the minimal parametric network is the point of a proper local minimum of the function $h$, therefore all the first partial derivatives $\d h/\d z_i^j=\d\ell_G/\d z_i^j$ are equal to zero at this point, and the second differential is a symmetric bilinear form. This form is non-degenerate, if and only if the corresponding quadratic form $\cQ$ is non-degenerate, and the value of the latter one at an arbitrary vector $\xi$ can be calculated as follows:
$$
\cQ(\xi)=\frac{d^2}{d t^2}\Big|_{t=0}h(Z+t\xi).
$$
The second derivative of the length of a straight segment under a linear deformation is calculated in Assertion~\ref{ass:deriv1} which implies that this derivative is non-negative. Therefore, $\cQ(\xi)=0$, if and only if the second derivative of the length vanishes for each edge of the network. Show that the latter is impossible.

Consider an arbitrary mustaches of the binary tree $G$, and let $z$ be the common vertex which the edges $zw$ and $zw'$ of the mustaches meet at. By Assertion~\ref{ass:deriv1}, the second derivative of the length of a straight segment is equal to zero, if and only if the difference of the deformation velocity vectors at its ends is parallel to the segment itself. The deformations under consideration remain fixed the boundary vertices, therefore the deformation velocity vectors are equal to zero at the points $w$ and $w'$. But the deformation velocity vector at the vertex $z$ can not be  parallel to the straight segments $zw$ and $zw'$ simultaneously, because the angle between these segments is equal to $2\pi/3$ in accordance with Assertion~\ref{ass:Eucl_exist}. Hence, the second derivatives of the lengths of the straight segments $zw$ and $zw'$ can not vanish simultaneously. Thus, $\cQ(\xi)\ne 0$ for any non-zero vector $\xi$. Lemma is proved.
\end{proof}

Lemma~\ref{lem:second_deriv} implies that the equations system $\d\ell_G/\d z_i^j=0$, $i=1,\ldots,n-2$, $j=1,\ldots,k$, satisfies the conditions of Proposition~\ref{prop:explicit_analitic_fucntion}, therefore it is solvable in a neighborhood of the point $M_0$, namely, there exist uniquely defined analytic functions $z_i=z_i(m_1,\ldots,m_n)$, $i=1,\ldots, n-2$, determining the location of the interior vertices of the minimal parametric network $\G$ of the type $G$ with the boundary $M$.  Assertion is proved.
\end{proof}

\begin{cor}\label{cor:lmn_analit}
Assume that all the edges of a minimal parametric binary tree $\G$ with a boundary $M$ are non-degenerate. Then each analytic deformation  $M_t$ of the boundary set generates an analytic deformation $\G_t(V)$ of the whole vertex set of the minimal parametric network $\G_t$ of the type $G=(V,E)$ with the boundary  $M_t$.
\end{cor}

\begin{cor}\label{cor:smt}
Assume that all the shortest trees connecting a boundary set  $M=\{m_1,\ldots,m_n\}\subset\R^k$ are non-degenerate. Let a one-parametric deformation $M_t$ of the boundary set be given, such that each curve $m_i(t)$ is analytic at the point $t=t_0$, $M_{t_0}=M$. Then
 \begin{enumerate}
\item There exists a neighborhood $U\ss[a,b]$ of the point $t_0$, such that for all $t\in U\cap\{t>t_0\}$ \(for all $t\in U\cap\{t<t_0\}$, respectively\/\) the types of the shortest trees connecting the sets $M_t$ are the same, and these types are contained  among the types of the shortest trees connecting the initial set $M_{t_0}$\/\rom;
\item If also the sets of types of the shortest networks connecting $M_t$ and $M_{t_0}$ are the same for some $t\in U\sm\{t_0\}$, then these sets coincide with each other for all $t\in U$.
\end{enumerate}
\end{cor}

\subsection{Stable Minimal Parametric Binary Trees and Stable Shortest Trees}
Fix the notations from the previous Subsection, and consider an arbitrary minimal parametric binary tree $\G$ of a type $G$ connecting a finite set $M\ss\R^k$. Consider the trace $\D$ of the network $\G$. Its regular components satisfy the conditions of Assertion~\ref{ass:int_vertices}, therefore, the interior vertices of each such component depend analytically on its boundary vertices. And if the regular components meet each other at the vertices of degree two only, then $\cB(\D)=\{G\}$, and so the following result holds.

\begin{cor}\label{cor:ar_lmn_analit_deg2}
Assume that the regular components of a shortest tree $\G$ with a boundary $M_{t_0}=M$ meet each other at boundary vertices of degree two only, and all the angles between the corresponding segments are strictly greater than  $2\pi/3$. Then for any sufficiently small analytic deformation $M_t$ of the boundary the decomposition of the minimal parametric network $\G_t$ of the type $G=(V,E)$ with the boundary $M_t$ into regular components is the same for all $t$ and $\G_t(V)$ is an analytic deformation of the whole vertex set of the network $\G_t$.
\end{cor}

A shortest tree $\G$ is said to be {\em stable}, if its regular components meet each other at the boundary vertices of degree two only and all the angles between the corresponding segments are strictly greater than  $2\pi/3$. Corollary~\ref{cor:ar_lmn_analit_deg2} and Theorem~\ref{thm:main} imply the following result.

\begin{cor}\label{cor:smt_gen}
Let $M_t$ be a one-parametric deformation of a boundary set $M=\{m_1,\ldots,m_n\}\subset\R^k$, such that each curve $m_i(t)$ is analytic at the point $t=t_0$, $M_{t_0}=M$, and let all the shortest trees connecting the set  $M$ be stable. Then
 \begin{enumerate}
\item There exists a neighborhood $U\ss[a,b]$ of the point $t_0$, such that the binary types of the traces of the shortest networks connecting the sets $M_t$ are the same for all $t\in U\cap\{t>t_0\}$ \(for all $t\in U\cap\{t<t_0\}$, respectively\/\), and these types are contained among the binary types of the traces of the shortest trees for $M_{t_0}$\/\rom;
\item If also the binary types of the traces of the shortest trees connecting $M_t$ and $M_{t_0}$ are the same for some $t\in U\sm\{t_0\}$, then they are the same for all $t\in U$.
\end{enumerate}
The length of shortest tree connecting $M_t$ is analytic on $t$.
\end{cor}

\section{Minimal Fillings of Finite Metric Spaces}
\markright{\thesection.~Minimal Fillings of Finite Metric Spaces}
The problem on minimal fillings of finite metric spaces appeared in~\cite{ITMatSb} as a result of a synthesis of two classical problems, namely, Steiner problem on shortest networks and Gromov problem on minimal fillings of Riemannian manifolds, see~\cite{Gromov}. The details can be found in~\cite{ITMatSb}, and here we include necessary concepts and results only. 

Let $G=(V,E)$ be a graph and $\om\:E\to\R$ be an arbitrary function on its edge set, which is usually referred as a  \emph{weight function}. In this case the pair $(\G,\om)$ is called a \emph{weighted graph}. The value of a weight function on an edge is called the \emph{weight\/} of this edge. For each subgraph $H$ in $G$ its  \emph{weight  $\om(H)$} is defined as the sum of all its edges. Similarly, the weight $\om(\g)$ of each rout $\g$ is defined.  If the graph $G$ is connected and the function $\om$ is non-negative, then a semi-metric $d_\om$ arises on the vertex set $V$ of the graph $G$, namely, the value $d_\om(x,y)$ is equal to the least possible weight of a rout in $G$ connecting $x$ and $y$. 

Let $(X,\r)$ be a semi-metric space, and $G=(V,E)$ be a connected  graph with a boundary $X$ and a non-negative weight function $\om$. Weighted graph  $(G,\om)$ is called a  \emph{filling of the space $(X,\r)$}, if for any points $x$ and $y$ from $X$ the inequality $\r(x,y)\le d_\om(x,y)$ is valid. The value  
$$
\mpf(X,\r,G)=\inf\bigl\{\om(G)\mid \om : \text{$(G,\om)$ is a filling of the space $(X,\r)$}\bigr\}
$$
is called the \emph{weight of minimal parametric filling of the type $G$} of the space $X$, and each weighted graph $(G,\om)$ which the infimum is attained at is called a \emph{minimal parametric filling of the type $G$} of the space  $X$. Further, the value 
$$
\mf(X,\r)=\inf\bigl\{\mpf(X,\r,G)\mid \text{$G$ is a connected graph with the boundary $X$} \bigr\}
$$ 
is called the \emph{weight of minimal filling of the space $X$}, and each weighted graph $(G,\om)$ which the infimum is attained at is called a \emph{minimal filling\/} of the space $X$.

In paper~\cite{ITMatSb} it is shown that for any semi-metric space $(X,\r)$ and any connected graph with the boundary $X$ there exists a minimal parametric filling of the type $G$, and for any semi-metric space $(X,\r)$ there exists a minimal filling. Moreover, among the minimal fillings one can always find a minimal filling, whose type is a binary tree with the boundary $X$. If the space $(X,\r)$ is assumed to be a metric space in addition, then there also exists a minimal fillings, whose type is a tree and whose weight function is strictly  positive.  

Since the number of binary trees with a boundary consisting of a fixed number of points is finite, then the problem of minimal filling of a finite metric space finding can be reduced to a finite enumeration of minimal parametric fillings (but an exponential one), each of which (i.e., the corresponding weight function in fact) can be found by linear programming. But it turns out that the weight of minimal filling of a space $(X,\r)$ can be expressed as a combinatorial formula also, that represents it as a function on the distances between the points from $X$.  

An existence and possible form of such formula were conjectured in~\cite{ITMatSb}. Later on it turns out that to obtain a correct expression one need to generalize the concept of a parametric filling in the case of trees permitting negative weights of edges~\cite{IOST}. Namely, a weighted tree $(T,\om)$ with the boundary $X$ is called a  \emph{generalized filling of a semi-metric space $(X,\r)$}, if for any pair of points $x$ and $y$ from $X$ the inequality $\r(x,y)\le \om(\g_{xy})$ is valid, where $\g_{xy}$ is the unique path in the tree $T$ connecting $x$ and $y$. The value 
$$
\mpf_-(X,\r,T)=\inf\bigl\{\om(T)\mid \om : \text{$(T,\om)$ is a generalized filling of the space $(X,\r)$}\bigr\}
$$
is called the \emph{weight of generalized minimal parametric filling of the type $T$} of the space $X$, and each weighted tree $(T,\om)$ which the infimum is attained at is referred as a \emph{generalized minimal parametric filling of the type $T$} of the space $X$. Further, the value 
$$
\mf_-(X,\r)=\inf\bigl\{\mpf_-(X,\r,T)\mid \text{$T$ is a tree with the boundary $X$} \bigr\}
$$ 
is called the \emph{weight of generalized minimal filling of the space $X$}, and each weighted tree $(T,\om)$ which the infimum is attained at is called a  \emph{generalized minimal filling\/} of the space $X$. 

It is not difficult to construct an example of a metric space $(X,\r)$ and a tree $T$ such that $\mpf_-(X,\r,T)<\mpf(X,\r)$. But as it is shown in paper~\cite{IOST}, $\mf_-(X,\r)=\mf(X,\r)$, therefore, the weight of minimal filling can be calculated as the minimal weight of generalized minimal fillings whose types are binary trees.

The final combinatorial formula for the weight of generalized minimal parametric filling is obtained in~\cite{Eremin} in terms of so-called tours~\cite{ITMatSb}  and multi-tours. Here we list the corresponding definitions.  

Let $S$ be a finite set consisting of $n$ elements. By a  \emph{multi cyclic order of multiplicity $k$ on the set $S$} we call a mapping $\pi\:\Z_{nk}\to S$ such that
\begin{enumerate}
\item $\pi(j)\ne\pi(j+1)$ for any $j\in\Z_{nk}$, and
\item for any element $s\in S$ its pre-image under the mapping $\pi$ consists of $k$ elements exactly.
\end{enumerate}
The value 
$$
p(X,\r,\pi)=\frac1{2k}\sum_{j\in\Z_{nk}}\r\bigl(\pi(j),\pi(j+1)\bigr)
$$
is called the \emph{multi-perimeter of the space $(X,\r)$ with respect to the multi cyclic order $\pi$} 

Let $T=(V,E)$ be a tree with a boundary $M$. For each its edge $e$ the forest $\bigl(V,E\setminus\{e\}\bigr)$ consists of two subtrees $T_1$ and $T_2$. Put $M_i=M\cap T_i$, $i=1,\,2$.  A multi cyclic order on $M$ is called a \emph{multi-tour of the tree $T$}, if there exists $k$ such that for any $e\in E$ and each $M_i$ there exist exactly  $k$ elements  $p\in\Z_{nk}$ such that $\pi(p)\in M_i$, but $\pi(p+1)\notin M_i$. This~$k$ is called the \emph{multiplicity of the multi-tour\/}; a multi-tour of multiplicity  $k$ is also referred as a \emph{$k$-tour}. It is clear that if a multi cyclic order is a multi-tour, then its multiplicity as the one of a multi cyclic order coincides with its multiplicity of a multi-tour. By  $\cO(T)$ we denote the set of all multi-tours of the tree $T$.

Notice that since any pair of vertices in the tree $T$ is connected by the unique path, then each $k$-tour $\pi$ of the tree $T$ defines the set of non-empty paths $\g_j$, $j\in\Z_{nk}$, in the tree $T$ connecting its boundary vertices $\pi(j)$ and $\pi(j+1)$, $j\in\Z_{nk}$. Each edge of the tree $T$ belongs to $2k$ such paths exactly. The union of all these $nk$ paths forms an Euler cycle in the graph obtained from $T$ by changing of each its edge by the family of $2k$ multiple edges, and the multi-tour can be considered as a walk along this Euler cycle by  the consecutive paths $\g_j$.

In paper~\cite{Eremin} the following formula
$$
\mpf_-(X,\r,T) = \max_{\pi\in\cO(T)}p(X,\r,\pi)
$$
is proved, and thus,
$$
\mf(X,\r) = \min\limits_{T}\max\limits_{\pi\in\cO(T)} p(X,\r,\pi),
$$
where the minimum is taken over all binary trees $T$ with the boundary $X$.

An essential defect of those formulas is that the set $\cO(T)$ which the maximum is taken over is infinite. To discard this defect so-called  \emph{irreducible\/} multi-tours are defined~\cite{Eremin}. Notice that for any two multi-tours $\pi$ and $\s$ of the tree $T$ there sum $\pi+\s$ is naturally defined as the consecutive walk along the corresponding Euler cycles. In particular, for any positive integer $n$ the multi-tour  $n\,\pi$ is defined. A multi-tour $\pi$ is called  \emph{irreducible}, if for any positive integer $m$ the multi-tour $m\,\pi$ can not be decomposed into a non-trivial sum of multi-tours, namely, if $m\,\pi=\pi_1+\pi_2$, then $\pi_i=m_i\pi$, $i=1,\,2$, and $m_1+m_2=m$. It can be shown that a binary tree with $n$ boundary vertices has at most $C_{C_n^2}^{2n-3}$ irreducible multi-tours, in particular, the set $\cO_n(T)$ of all irreducible multi-tours of an arbitrary binary tree $T$ is finite. The following result holds, see~\cite{Eremin}.

\begin{ass}\label{ass:mf_formula}
For an arbitrary finite semi-metric space $(X,\r)$ and an arbitrary binary tree $T$ with the boundary $X$ the weight of generalized minimal parametric filling of the type $T$ can be calculated as the following maximum over a finite set of linear functions on the distances between the sets from the space $X${\rm :}
$$
\mpf_-(X,\r,T) = \max_{\pi\in\cO_n(T)}p(X,\r,\pi). 
$$
The weight of minimal filling can be calculated as the following finite minimax\/{\rm :} 
$$
\mf(X,\r) = \min\limits_{T}\max\limits_{\pi\in\cO_n(T)} p(X,\r,\pi),
$$
where the minimum is taken over all the binary trees $T$ with the boundary $X$.
\end{ass}

Theorem~\ref{thm:main}, Assertion~\ref{ass:mf_formula} and analyticity of linear functions imply the following result. 

\begin{cor}\label{cor:mf_param}
Let $(X,\r)$ be an arbitrary semi-metric space, and $\r_t$, $t\in[a,b]$, be a deformation of the semi-metric $\r=\r_{\t_0}$, which is analytic at $t=t_0$. Then 
\begin{enumerate}
\item for any binary tree $T$ with the boundary $X$ there exists a neighborhood  $U$ of the point $t_0\in[a,b]$ such that the sets of multi-tours of the tree $T$ which the weight of generalized minimal parametric filling of the type $T$ of the space $(X,\r_t)$ is attained at are the same for all $t\in U\cap\{t>t_0\}$ \(for all $t\in U\cap\{t<t_0\}$, respectively\) and these types are contained in the set of such multi-tours for $t=t_0$\/\rom;
\item if these sets of multi-tours coincide with each other for some $t\in U\sm\{t_0\}$ and $t_0$ in addition, then these sets are the same for all $t\in U$.
\end{enumerate}
\end{cor}

\begin{cor}\label{cor:mf_gen}
Let $(X,\r)$ be an arbitrary semi-metric space and $\r_t$, $t\in[a,b]$, be a deformation of the semi-metric $\r=\r_{\t_0}$, which is analytic at $t=t_0$. Then 
\begin{enumerate}
\item there exists a neighborhood $U$ of the point $t_0\in[a,b]$ such that the sets of types of minimal fillings of the space $(X,\r_t)$ are the same for all $t\in U\cap\{t>t_0\}$ \(respectively, for all $t\in U\cap\{t<t_0\}$\), and all these sets are contained in the set of types of minimal filings of the space  $(X,\r_{t_0})$\/\rom;
\item if these sets of types coincide for some $t\in U\sm\{t_0\}$ and $t_0$ in addition, then these sets are the same for all  $t\in U$.
\end{enumerate}
\end{cor}

\begin{cor}\label{cor:mf_eucl}
Let $M_t$, $t\in[a,b]$, be a one-parametric deformation of a finite subset $M_{t_0}=\{m_1,\ldots,m_n\}$ of the space $\R^k$ such that each point $m_i$ moves along the corresponding curve $m_i(t)$ that is analytic at the point  $t_0$. At each set $M_t$ we consider the metric induced from $\R^k$. Then 
\begin{enumerate}
\item there exists a neighborhood $U$ of the point $t_0\in[a,b]$ such that the sets of types of minimal fillings of the space $M_t$ are the same for all $t\in U\cap\{t>t_0\}$ \(respectively, for all $t\in U\cap\{t<t_0\}$\), and all these sets are contained in the set of types of minimal filings of the space  $M_{t_0}$\/\rom;
\item if these sets of types coincide for some $t\in U\sm\{t_0\}$ and $t_0$ in addition, then these sets are the same for all  $t\in U$.
\end{enumerate}
\end{cor}

\section{Segment's Length Derivatives under Linear Deformations}
\markright{\thesection.~Segment's Length Derivatives under Linear Deformations}
Let  $AB$ be an arbitrary non-degenerate straight segment in the Euclidean space $\R^k$. Consider its one-parametric linear deformation $A(t)B(t)$, where $A(t)=A+u\,t$ and $B(t)=B+v\,t$, and the length function $\ell(t)=\|A(t)B(t)\|$. Put $x=B-A$ and $w=v-u$, then $\ell(t)=\sqrt{\<x+w\,t,x+w\,t\>}$.

\begin{ass}\label{ass:deriv1}
If $\ell(t)$ is non-zero, then the function $\ell$ is infinitely differentiable at the point $t$, and under the above notations its first and second derivatives have the form
$$
\ell'(t)=\frac{\<w,x+w\,t\>}{\|x+w\,t\|},\qquad \ell''(t)=\frac{\<w,w\>\<x+w\,t,x+w\,t\>-\<w,x+w\,t\>^2}{\|x+w\,t\|^3}.
$$
In particular,
$$
\ell'(0)=\<w,\tau\>,\qquad \ell''(0)=\frac{\<w,w\>\<x,x\>-\<w,x\>^2}{\|x\|^3}=\frac{\<w,\nu\>^2}{\|x\|},
$$
where $\tau=(B-A)/\|AB\|$ is the unit vector of the segment's direction, and $\nu$ is any unit vector orthogonal to the segment. In particular, the second derivative is always non-negative.
\end{ass}

Now consider a two-parametric linear deformation $A(t)B(s)$ of the segment $[A,B]$, where $A(t)=A+u\,t$, $t\in[-\e,\e]$ and $B(s)=B+v\,s$, $s\in[-\e,\e]$, $\e>0$, and the length function $\ell(s,t)=\|A(t)B(s)\|$. Put $x=B-A$, then $\ell(s,t)=\sqrt{\<x+v\,s-u\,t,x+v\,s-u\,t\>}$.

\begin{ass}\label{ass:deriv2}
If the value $\ell(s,t)$ is non-zero, then the function $\ell$ is infinitely differentiable at the point $(s,t)$, and under the above notations its first partial derivatives  have the form
$$
\frac{\d\ell}{\d t}=-\frac{\<u,x-u\,t+v\,s\>}{\|x-u\,t+v\,s\|},\qquad \frac{\d\ell}{\d s}=\frac{\<v,x-u\,t+v\,s\>}{\|x-u\,t+v\,s\|}.
$$
In particular,
$$
\frac{\d\ell}{\d t}(0,0)=-\<u,\tau\>,\qquad \frac{\d\ell}{\d s}(0,0)=\<v,\tau\>,
$$
where  $\tau=(B-A)/\|AB\|$ is the unit vector of the segment's direction.

The second partial derivatives at the point $(0,0)$ have the form
$$
\frac{\d^2\ell}{\d t^2}(0,0)=\frac{\<x,x\>\<u,u\>-\<u,x\>^2}{\|x\|^3},\qquad
\frac{\d^2\ell}{\d s\,\d t}(0,0)=\frac{\<x,u\>\<x,v\>-\<x,x\>\<u,v\>}{\|x\|^3},
$$
and the formula for $\d^2\ell/\d s^2$ can be obtained from the one for $\d^2\ell/\d t^2$ by changing $u$ to $v$.  The numerators of these expressions are the minors of the Gram matrix of the vectors $\{x,u,v\}$, and the minors corresponding to $\d^2\ell/\d t^2$ and $\d^2\ell/\d s^2$ are principal, and hence, are non-negative.
\end{ass}

\section{Multidimensional Generalizations}
\markright{\thesection.~Multidimensional Generalizations}
Instead of the edge sets one can consider simplicial sets changing the family $V^{(2)}$ by $V^{(k)}$, $k>2$. Then one can consider metric functionals corresponding to the volumes of the Euclidean simplices using Cayley--Menger determinants, instead of edges lengths. For such functionals an analogue of Theorem~\ref{thm:main} also holds.

\section*{Acknowledgments}
The authors are faithfully appreciative to A.~Fomenko for his permanent attention to their work, and also to H.~Rubinstein for fruitful discussions that force the authors to become more attentive to the problems of smooth and analytic deformations. The work is partially supported by  RFBR and RF President Project of Russian Scientific Schools support. 

\markright{References}

\end{document}